\documentclass[11pt,letterpaper]{amsart} 
\usepackage{amssymb,latexsym}
\usepackage{amsmath, amscd}
\usepackage[all]{xy}

\def\f2{\ensuremath{\mathbb F_2}}
\def\R{\ensuremath{\mathbb R}}
\def\H{\ensuremath{\mathcal H}}
\def\N{\ensuremath{\mathcal N}}
\def\SH{\ensuremath{\mathcal S^h}}
\def\SS{\ensuremath{\mathcal S^s}}
\def\SHS{\ensuremath{\mathcal S^{hs}}}

\newcommand\eg{e.\,g.\,}
\newcommand\ie{i.\,e.\,}
\renewcommand\mod{\text{ mod\ }}

\def\Z{\mathbb Z}
\def\Q{\mathbb Q}

\newtheorem{thm}{Theorem}[section]
\newtheorem{lm}[thm]{Lemma}

\newtheorem{cor}[thm]{Corollary}

 \theoremstyle{remark}
\newtheorem{ex}[thm]{Example}
\newtheorem{rem}[thm]{Remark}
\newtheorem{defn}[thm]{\bf Definition}

\begin{document}

\title{How different can $h$-cobordant manifolds be?}


\author{Bj\o rn Jahren}
\address{Dept. of Math., University of Oslo, Norway}
\email{bjoernj@math.uio.no}
\author{S\l awomir Kwasik}
\address{Dept. of Math., Tulane University, New Orleans, LA, USA}
\email{kwasik@math.tulane.edu}

\begin{abstract} We study the homeomorphism types of manifolds $h$-cobordant to a fixed one.  Our investigation
is  partly motivated by the notion of {\em special manifolds} introduced by Milnor in his study of lens spaces. In particular 
we revisit and clarify some of the claims concerning  $h$-cobordisms of these manifolds.
\end{abstract}

\maketitle

\section{Introduction}

The problem of recognizing whether two homotopy equivalent manifolds $M$ and $N$ are homeomorphic, 
 diffeomorphic or PL-isomorphic, depending on the category we work in, is of fundamental
importance in modern  geometric topology. Typically, in each category this problem is approached in three steps:
\begin{enumerate}
\item[i.] Showing that $M$ and $N$ are cobordant
\item[ii.] Improving the cobordism to an $h$-cobordism
\item[iii.] Showing that the Whitehead torsion of the constructed $h$-cobordism is trivial.
\end{enumerate}
The $s$-cobordism theorem of Barden-Mazur-Stallings then yields an equivalence between $M$ and $N$,
provided the dimension is at least 5, see \cite{mazur}, \cite{KiSi}, \cite{R}.\par
Mainly for convenience we will work with topological manifolds in this paper.  See, however, section \ref{sec:diff} for comments 
on the other categories.\par
Showing that the Whitehead torsion is trivial is often the most difficult step in the above program.
However, the triviality of the Whitehead torsion is not always necessary for $M$ and $N$ to be homeomorphic:  there are many examples of non-trivial $h$-cobordisms such that the ends are homeomorphic.  
We call $h$-cobordisms with homeomorphic ends {\em inertial}, and  a central problem is to determine the subset of elements of the 
Whitehead group $Wh(\pi_1(M))$ that  can be realized as Whitehead torsions of inertial $h$-cobordisms. This is 
in general very difficult, and only partial results in this direction are known --- see 
\eg \cite{mwt}, \cite{H75, H80}, \cite{L74, L77}.\par
The following  important observation was made by Hatcher and Lawson in \cite{HL}:\medskip

\noindent{\bf Stability criterion.} {\em Let $M$ be a manifold of dimension $n\geqslant 5$ with fundamental group $G$.  If $\tau\in Wh(G)$ can
 be represented by a $d\times d$-matrix, then every $h$-cobordism with one end homeomorphic to
$M\#_d (S^p\times S^{n-p})$ is inertial, if $2\leqslant p\leqslant n-2$.}\bigskip

(Here $M\#_d (S^p\times S^{n-p})$ denotes the connected sum of $M$ and $d$ copies of products of 
spheres $S^p\times S^{n-p}$.)\par 
Since  Whitehead groups of {\em finite} groups are much better understood than those of infinite groups,  
in what follows we will mostly concentrate on the finite group case. 
For example, it turns out that if $G$ is finite, we can choose $d=2$ for all $\tau$ in the above proposition \cite{V}, and it 
follows that  {\em every} $h$-cobordism from $M\#(S^p\times S^{n-p})\#(S^p\times S^{n-p})$ is inertial.
\smallskip

On the other hand, results by Milnor for lens spaces \cite{mwt} and \cite{KS99} for  general spherical space forms show 
that for these manifolds, no non-trivial $h$-cobordism is inertial. Hence, for some manifolds, stabilization by 
one or two copies of $S^p\times S^{n-p}$ is necessary for $h$-cobordant manifolds to become 
homeomorphic, except in the the trivial case. In fact, for lens spaces (and also many other space forms), 
one copy of $S^p\times S^{n-p}$ will always suffice. More generally, by Hatcher-Lawson's stability criterion the same is true for 
all manifolds with fundamental groups such that every element in the Whitehead group is represented 
by units in the group ring.  This argument breaks down for general groups, but to the best of our knowledge, 
until now no example has been known where one copy of $S^p\times S^{n-p}$ in the stabilization is not enough.  Therefore 
a starting points of our investigation was the following question:
\begin{quote}
Is there a manifold $M$ with finite fundamental group and  an $h$-cobordism
from $M \#(S^p\times S^{n-p})$ which is not inertial?
\end{quote} 
One of our main results is a positive  answer to this question, in the following equivalent form:

\begin{thm}\label{thm:noninertial}
In every  dimension $n=4k+3\geqslant 7$ we can find  $h$-cobordant manifolds $M$ and $N$ with finite fundamental groups such that 
$M\#(S^p\times S^{n-p})$ and $N\#(S^p\times S^{n-p})$ are not homeomorphic for any $p$ such that 
$2\leqslant p\leqslant n-2$.
\end{thm}

In contrast to the problem of inertial $h$-cobordisms one can also ask how often realization of Whitehead torsion leads to
{\em non}--inertial $h$-cobordisms, \ie how many different homeomorphism classes of manifolds  that are $h$-cobordant to
a given manifold $M$.  For example, Milnor proved that for even--dimensional manifolds with finite fundamental group, only finitely
many homeomorphism classes can be realized \cite[Thm. 11.5]{mwt}, but   in the odd case he proves that there sometimes 
can be infinitely many \cite[Cor. 12.9]{mwt}.  We will generalize this result and prove

\begin{thm}\label{thm:infex}Let $G$ be a finite group such that $Wh(G)$ is infinite.  Then in every odd dimension $\geqslant 5$ there are manifolds $M$ with
fundamental group $G$  such that there are infinitely many distinct homeomorphism classes of manifolds $h$-cobordant to $M$.
\end{thm}
This theorem and its proof are motivated by Milnor's notion of {\em special manifolds} \cite[\S 12]{mwt}, and we show that between 
manifolds satisfying certain conditions there are only finitely many different $h$-cobordisms. For some manifolds $M$ 
(\eg spherical space forms) a much stronger statement is true:  there is at most one such 
$h$-cobordism between two of them.    However, we  also show (end of Section \ref{sec:abelian}) that there are groups such that
from any odd-dimensional manifold with these fundamental groups there are non-trivial inertial $h$-cobordisms.   Thus the conclusion of Theorem \ref{thm:infex} is in some sense best possible.\par
The notion of special manifolds was based on an algebraic assumption that
turned out to be incorrect (see Remark after Example 1.6 of \cite{mwt}), although it holds for cyclic groups, which were
Milnor's main examples.  In Section \ref{sec:special} we make some more comments on this, and we suggest a variant
 that circumvents the algebraic problem, thus rendering Milnor's discussion basically correct, with
appropriate modifications.\par
Milnor's special manifolds were assumed to have finite abelian fundamental groups, and the most striking assertions were for
the odd-dimensional case. This prompted us to look closer at $h$-cobordisms between  {\em arbitrary} orientable, odd-dimensional
manifolds with a finite abelian fundamental group.  In view of examples \ref{ex:lens} and  \ref{ex:K2} below, our Theorem \ref{thm:abelian},
asserting the triviality of "strongly inertial" such $h$-cobordisms, seems to be the best and most general substitute
for Theorem 12.8 in \cite{mwt}. 

\section{General remarks on  $h$-cobordism and torsion}\label{sec:general}

We use the notation $(W;M,M')$ for an $h$-cobordism $W$ with boundary manifolds $M$ and $M'$. Recall that this means that 
$W$ is a manifold with two boundary components $M$ and $M'$, each of which is a deformation retract of $W$. More 
specifically, we will
think of this as an $h$-cobordism {\em from} $M$ {\em to} $M'$, thus distinguishing it from the {\em dual} $h$-cobordism
$(W;M',M)$. Since the pair $(W,M)$ obviously determines $M'$, we will also sometimes use the simpler notation $(W,M)$ for 
$(W;M,M')$. One may allow $M$ to have boundary and require everything to be fixed along  $\partial M$, but to keep the 
notation simpler  we will assume $M$ and $M'$  closed in this paper.\par 
Let $\mathcal H(M)$ be the set of homeomorphism classes relative $M$ of $h$-cobordisms from $M$.\par

For a path connected space $X$, denote by $Wh(X)$ the Whitehead group $Wh(\pi_1(X))$.  Note that this is 
independent of choice of base point of $X$, up to unique isomorphism. 
The $s$-cobordism theorem says that if $M$ is compact, connected and of dimension at least 5 there is a one-one correspondence between $\H(M)$ and $Wh(M)$, associating to the $h$-cobordism $(W;M,M')$ the Whitehead torsion $\tau(W,M)\in Wh(M)$ of the pair of spaces.  Note that Milnor \cite{mwt} places the torsion in the canonically isomorphic group $Wh(W)$, but as our emphasis will
be on the set of $h$-cobordisms from a fixed manifold $M$, it is more natural for us to follow the convention in \cite{cohen}.\par
One advantage of having all the torsion elements in the same group, is that it makes it easier to give a geometric description of the group operation in $Wh(M)$. 
Given an element $(W;M,M')\in \mathcal H(M)$, define a homotopy equivalence $h:M'\to M$ as the composition of the inclusion
$M'\subset W$ and a retraction $W\to M$.  The homotopy class of $h$ is uniquely defined, hence so is also the induced isomorphism
$h_*:Wh(M') \to Wh(M)$. Even though $h$ is only well-defined up to homotopy, we will refer to it as the "natural homotopy 
equivalence" asociated to $(W,M)$ (or $(W;M,M')$).\par
 Now let $\tau$ and $\sigma$ be two elements
of the Whitehead group, and represent $\tau$ as $\tau(W;M,M')$ for an $h$-cobordism $(W;M,M')$.  Similarly, represent 
the element $h_*^{-1}(\sigma)\in Wh(M')$ as $\tau(W';M',M'')$, and let $W''=W\cup W'$, glued along $M'$.  Then 
\begin{lm}\label{lm:add}
$\tau(W'';M,M'')=\tau+\sigma.$
\end{lm}
\begin{proof} This follows from (20.2) and (20.3)  in \cite{cohen}.
\end{proof}
Henceforth we will often tacitly identify $\H(M)$ with the abelian group $Wh(M)$.\smallskip

Recall that there is an involution $\tau\mapsto \bar \tau$ on $Wh(M)$, induced by transposition of matrices and inversion of 
group elements \cite[\S 6]{mwt}.  
 An observation of great importance to us is the following relation between the Whitehead torsions of $(W;M,M')$ and $(W;M',M)$, 
where $M$ is orientable and has dimension $n$ \cite[p.\,394]{mwt}:
\begin{equation}\label{eq:dual}
h_*(\tau(W,M'))=(-1)^n\overline{\tau(W,M)}.
\end{equation}

A consequence of (\ref{eq:dual}) and Lemma \ref{lm:add} is the usual formula for the torsion of the {\em double} of 
an $h$-cobordism $(W;M,M')$ with torsion $\tau$: 
\begin{equation}\label{eq:double}
\tau(W\cup_{M'}W)=\tau+(-1)^n\bar \tau.
\end{equation}

\section{Inertial sets}

We are principally interested in computing the {\em inertial set} of a manifold $M$, defined as
$$I(M)=\{(W;M,M')\in \H(M)\,|\, M'\approx M\},$$
or the corresponding subset of $Wh(M)$.
As already noted by Hausmann \cite[Remark 6.2]{H80}, this set does not have good algebraic properties.  In particular, it
is not a subgroup of $Wh(M)$. From Lemma \ref{lm:add} we see that this means that the induced maps $h_*$ in general do not preserve inertial sets. However, there is a smaller set which does behave better.\par

\begin{defn}  
 The $h$-cobordism $(W;M,M')$ is called {\em strongly inertial} if the natural homotopy equivalence $h:M'\to M$ is homotopic
to a homeomorphism.

The set of strongly inertial $h$-cobordisms from $M$ is denoted $SI(M)$. 
\end{defn}

 It is easy to see directly that this is a subgroup of
$Wh(M)$, but the following digression puts this into the perspective of surgery theory.\par

Recall the structure sets   $\SH(M)$  (or $\SS(M)$) given by  arbitrary (or simple)  homotopy equivalences  $f:N\to M$ modulo
$h$-cobordisms ($s$-cobordisms).  Define an intermediate structure set $\SHS(M)$ as the set of arbitrary homotopy
equivalence $f:N\to M$, modulo $s$-cobordisms. Then the forgetful map $\SS(M)\to \SH(M)$ factors as an inclusion
$\SS(M)\subset \SHS(M)$ followed by a surjection $\SHS(M)\to \SH(M)$. \par
Define an action of $Wh(M)$ on $\SHS(M)$ as follows:  Let $\tau\in Wh(M)$ and let $f:N\to M$ be a homotopy equivalence. Then
$f_*^{-1}(\tau)\in Wh(N)$ is the torsion of a unique $h$-cobordism $(W;N,N')$ which defines a natural 
 homotopy equivalence $h:N'\to N$.  Let $\tau\cdot [f]= [f\circ h]$.\par
It is easily checked that this defines an action with quotient set $\SH(M)$. Moreover, the isotropy subgroup of the trivial
element $id:M\to M$ is precisely $SI(M)$.\smallskip

For any $h$-cobordism $(W;M,M')$, the double $(W\cup_{M'}W;M,M)$ is obviously strongly inertial. Thus, from formula
(\ref{eq:double}) we have
\begin{equation}\label{eq:inertinc}\{\tau+(-1)^n\bar \tau|\tau\in Wh(M)\}\subset SI(M)\subset I(M).
\end{equation}
One of our aims is to examine how these sets differ.

\section{Stabilization and proof of Theorem \ref{thm:noninertial}}

The proof is based on following lemma, which can be seen as a partial converse to Hatcher-Lawson's stability criterion  in the case $d=1$.

\begin{lm}\label{lm:MNhcob}Let $M$ be a manifold of dimension  $n$ with finite fundamental group $G$, and assume that $\pi_p(M)=0$ for some $p$ with 
$2\leqslant p< n/2$. Suppose that $N$ is another manifold, such that $M\#(S^p\times S^{n-p})\approx N\#(S^p\times S^{n-p})$.   
Then $M$ and $N$ are $h$-cobordant by an $h$-cobordism with Whitehead torsion represented by a unit in $\Z(\pi_1(M))$.
\end{lm}

\begin{proof}  Let $f:N\#(S^p\times S^{n-p})\to M\#(S^p\times S^{n-p})$ be a homeomorphism. Let 
$W_1=(M\times I)\cup (D^p\times D^{n-p+1})$ be $M\times I$ with a $p$--handle attached trivially inside a disk in $M\times\{1\}$.
Then $W_1$ is a cobordism between $M$ and $M\#(S^p\times S^{n-p})$, and the pair $(W_1,M)$ is homotopy equivalent to 
$(M\vee S^p,M)$. Dually, we can also think of $W_1$ as obtained from  $M\#(S^p\times S^{n-p})$ by adding an $(n-p+1)$--handle to
 $(M\#(S^p\times S^{n-p}))\times I$ attached along a neighborhood of  $\{y\}\times S^{n-p}\times \{1\}$, for some $y\in S^p$.\par

 Similarly we get a cobordism 
$W_2=N\times I\cup D^{p+1}\times D^{n-p}$ from $N$  to $N\#(S^p\times S^{n-p})$ by adding a trivial $(n-p)$--handle to
$N\times I$, {\em or} a $(p+1)$--handle to $(N\#(S^p\times S^{n-p}))\times I$ . Now glue $W_1$ and $W_2$ using the 
homeomorphism $f$ to obtain a cobordism $W$ between $M$ and $N$.  We claim that under conditions as in the Lemma, $W$ is an $h$-cobordism with torsion represented by a unit.\par

In fact, as we have just observed, we can think of  $W$ as built up from $M\times I$ by
adding one handle of index $p$ and then one of index $p+1$. Since $p\geqslant 2$, all inclusion relations between the cobordisms
and their boundaries are  $\pi_1$--isomorphisms. Denote this common fundamental group by $G$. \smallskip

\begin{rem} Here, and several places in the following, there is an issue of choice of basepoint. Thus, the fundamental groups involved 
are not "the same", but it is not hard to see that basepoints can be chosen coherently such that there are canonical isomorphisms between the
groups.
More specifically, in $M$ and $N$ and the corresponding connected sums we can choose basepoints lying in the connecting $(n-1)$-spheres.
The homeomorphism $f$ can be isotoped to one preserving the basepoints, and then an obvious curve in  $W$ containing all the basepoints
can be used to compare the various groups in the standard manner. Finally, a choice of lifting of this curve to $\widetilde W$ gives compatible 
coherent choices of basepoints and isomorphisms for the universal covers.  Henceforth we choose to drop these choices from the notation and identify all the fundamental groups with $G=\pi_1(M)$.
\end{rem}

The relative homology of the pair $(\widetilde W,\widetilde M)$  of universal covering spaces can be computed (as $\Z G$-module) as the 
homology of the chain complex
\begin{equation}
\cdots\to 0\to C_{p+1}\xrightarrow{d_{p+1}} C_p\to0\to\cdots,
\end{equation}
where $C_p=H_p(\widetilde W_1,\widetilde M)$ and $C_{p+1}=H_{p+1}(\widetilde W,\widetilde W_1)$. These are both free $\Z G$-modules
of rank one with bases given by (liftings of) the cores of the two handles. Consequently, if we show that $d_{p+1}$ is an isomorphism, it is given by multiplication by a unit  $u\in \Z G$.  Then  $W$ will be an $h$-cobordism with Whitehead torsion represented by $u$, 
up to sign.\par

Recall that $d_{p+1}$ is a connecting homomorphism in the long, exact homology sequence of the triple $(\widetilde W, \widetilde W_1,
\widetilde M)$ and factors as
$$H_{p+1}(\widetilde W,\widetilde W_1)\to H_{p}(\widetilde W_1)\to H_{p}(\widetilde W_1,\widetilde M).$$

Via the excision isomorphism $H_{p+1}(\widetilde W,\widetilde W_1)\approx H_{p+1}(\widetilde W_2,\widetilde{N\#(S^p}\times S^{n-p}))$,
this is easily seen to be the same as the composition
\begin{multline} H_{p+1}(\widetilde W_2,\widetilde{N\#(S^p}\times S^{n-p}))\to H_{p}(\widetilde{N\#(S^p}\times S^{n-p}))\to\\
\xrightarrow{f_*}
H_{p}(\widetilde{M\#(S^p}\times S^{n-p}))\xrightarrow{i_*} H_p(\widetilde W_1)\to H_p(\widetilde W_1,\widetilde M),
\end{multline}
where $i$ is the obvious inclusion. We want to prove that the composite map is an isomorphism, although the intermediate maps may not be.
\par
The following diagram compares this with the corresponding maps on homotopy groups, via the Hurewicz homomorphisms for
the universal covering spaces:

\begin{multline}{\xymatrix{ H_{p+1}(\widetilde W_2,\widetilde{N\#(S^p}\times S^{n-p}))\ar[r] & H_{p}(\widetilde{N\#(S^p}\times S^{n-p}))\ar[r]& \\
\pi_{p+1}(W_2, N\#(S^p\times S^{n-p}))\ar[u]_{h_1}\ar[r]^{\delta} & \pi_p({N\#(S^p}\times S^{n-p}))\ar[u]\ar[r]&
}}\\
{\xymatrix{&\ar[r]^<<<<{f_*}&H_{p}(\widetilde{M\#(S^p}\times S^{n-p}))\ar[r]^>>>>{i_*}& H_p(\widetilde W_1)\ar[r]& H_p(\widetilde W_1,\widetilde M)\\
&\ar[r]^<<<<{f_*}& \pi_{p}({M\#(S^p}\times S^{n-p}))\ar[u]\ar[r]^>>>>{i_*} &\pi_p( W_1)\ar[r]^{j_*}\ar[u]& 
\pi_p(W_1,M)\,.\ar[u]_{h_2}
}}\end{multline}

First, note that since the pair $(\widetilde W_2,\widetilde{N\#(S^p}\times S^{n-p}))$ is  $p$-connected and 
$(\widetilde W_1,\widetilde M)$ is $(p-1)$-connected, the maps denoted $h_1$ and $h_2$ are isomorphisms, by the relative Hurewicz theorem.
Therefore we are done if we can prove that each of the lower horizontal maps are isomorphisms.\par

We start from the right,  with $j_*$.  Since $W_1\simeq M\vee S^p$ with $p\ge2$, the long exact homotopy sequence for the pair $(W_1,M)$ splits up into split, short exact pieces
$$0\to \pi_k(M)\to \pi_k(W_1)\xrightarrow{j_*} \pi_k (W_1,M)\to 0.$$
 By the assumption $\pi_p(M)=0$, $j_*$ is an isomorphism for $k=p$.\par

To see that $i_*: \pi_{p}({M\#(S^p}\times S^{n-p}))\to\pi_p( W_1)$ is an isomorphism, we  first observe that  we can think 
of $W_1$ as obtained from ${M\#(S^p}\times S^{n-p})$ by attaching an $(n-p+1)$-handle. Then the assertion follows since, 
by the assumption on $p$,  we have $n-p>p$.\par

Since $f$ is a homeomorphism, $f_*$ is automatically an isomorphism.\par
It remains to consider $\delta$. We now know that both source and target are free $\Z G$-modules of rank 1. Moreover, it is clear that 
we can choose a generator
of $\pi_{p+1}(W_2, N\#(S^p\times S^{n-p}))$ which is mapped to the element in $\pi_p({N\#(S^p}\times S^{n-p}))$ represented by
a map of the form
$$g:S^p=S^p\times\{y\}\subset N\#(S^p\times S^{n-p}),$$ 
for some $y\in S^{n-p}$. We need to show that $\pi_p({N\#(S^p}\times S^{n-p}))$ is generated by $[g]$.\par

Just as for $M\#(S^p\times S^{n-p})$ we have isomorphisms 
$$\pi_p({N\#(S^p}\times S^{n-p}))\approx \pi_p(N\vee S^p)\approx \pi_p(N)\times \pi_p(N\vee S^p, N)\approx \pi_p(N)\times \Z G,$$
with the $\Z G$--factor  generated by the image of $[g]$. But  this is an isomorphism of the form 
$\Z G\approx \pi_p(N)\times \Z G$, and it follows by the classification of 
finitely generated abelian groups that $\pi_p(N)$ must vanish. Hence $[g]$ generates  $\pi_p({N\#(S^p}\times S^{n-p}))$.
\end{proof}

To obtain examples of manifolds as in Theorem \ref{thm:noninertial}, we can now take $M$ to be a spherical space form with fundamental
group $G$ such that not every element of $Wh(G)$ is represented by a unit.  Examples of such groups include \eg quaternion groups 
$Q_{16p}$ with $p\equiv -1\mod 8$, or $Q_{4p}$ with $p$ a prime with even class number $h_p$ \cite[Thm. 10.8]{oliver}.  These groups are
fundamental groups of spherical space forms of every dimension $4k+3$, and spherical space forms of dimension at least 5 certainly satisfy the conditions of Lemma \ref{lm:MNhcob}. By \cite{KS99} they also have trivial  inertial sets $I(M)$, and we claim that two such manifolds which are $h$-cobordant with Whitehead torsion not represented by a unit cannot also be $h$-cobordant
with Whitehead torsion which is represented by unit.\par
To verify this claim, let $\tau_1= \tau(W_1;M,M_1)$  and $\tau_2= \tau(W_1;M,M_2)$ be such that $\tau_2$ is represented by a unit and $\tau_1$ is not, and suppose $f:M_1\to M_2$ is a homeomorphism.  Also, let $h_1:M_1\subset W_1\to M$ and 
$h_2:M_2\subset W_2\to M$ be the natural homotopy equivalences. 
Now we construct an inertial $h$-cobordism $(V;M,M)$ by gluing $W_1$ and $W_2$ along $M_1$ and $M_2$ using the homeomorphism $f$.  By the discussion in Section \ref{sec:general} the torsion of $V$ is 
$$\tau(V,M)=\tau_1+(-1)^n {h_1}_*f_*{h_2}_*^{-1}(\bar\tau_2),$$ 
which must be trivial by the choice of $M$.  Thus 
\begin{equation}\label{eq:tau12}\tau_1=-(-1)^n {h_1}_*f_*{h_2}_*^{-1}(\bar\tau_2).
\end{equation} 

But if $\tau_2$ is represented by a unit, then  $\pm{h_1}_*f_*{h_2^{-1}}_*(\bar\tau_2)$ is also represented by a unit.
Hence  equation (\ref{eq:tau12}) is impossible.
\qed

\section{Proof of Theorem \ref{thm:infex}}

The idea is to use $R$--torsion as in Milnor's proof of Theorem 12.9 in \cite{mwt}, applied to certain manifolds constructed using work 
by Pardon \cite{P}. 
We start with a short discussion of the aspects of $R$--torsion that we need.  \par
The $R$--torsion is an invariant defined for finite CW-complexes $X$ with finite fundamental group, such that $\pi_1(X)$ 
acts trivially on rational homology of the universal covering $\widetilde X$.  Write $\pi_1(X)=G$ and let $\Sigma\in \Q G$ be the 
sum of the group elements of $G$.   Then there is a splitting of rings $\Q G\approx (\Sigma)\oplus \Q_G$, where 
$(\Sigma)$ is the (two-sided) ideal generated by
$\Sigma$ and $\Q_G=\Q G/(\Sigma)$. This splitting induces splittings  
$$C(\tilde X)\approx \Sigma C\oplus C_G$$
of cellular chain complexes, where $C_G$ is acyclic by the assumption on $X$. It is also free and based over the ring $\Q_G$, and so
has a torsion well defined in $K_1(\Q_G)/(\pm G)$.  This is the $R$--torsion $\Delta(X)$.
We will denote $K_1(\Q_G)/(\pm G)$ by $Wh(\Q_G)$.  There is an obvious homomorphism from $Wh(G)$ to $Wh(\Q_G)$ which factors 
through an injection $Wh'(G)\to Wh(\Q_G)$, where, as usual, $Wh'(G)$ denotes the image of $Wh(G)$ in $K_1(\Q G)/(\pm G)$.  \par
Now suppose $Y$ is a space homotopy equivalent to $X$, and let $h:Y\to X$ be a homotopy equivalence. Then we can define the
$R$--torsions of $Y$ and $X$, and the same procedure gives a torsion class $\Delta(M_h,Y)\in Wh(\Q_G)$ which is the image of 
the torsion $\tau(h)\in Wh(G)$ of the homotopy equivalence $h$. We will keep the notation $\tau(h)$ for $\Delta(M_h,Y)$.
(Here $M_h$ is the mapping cylinder of $h$.)\par
Now we can compare the $R$--torsions, using \cite[Theorem. 3.1]{mwt}, applied to the short exact sequence of chain complexes
of $\Q_G$--modules
$$0\to C_{G}(Y)\to C_G(M_h)\to C_G(M_h,X)\to 0.$$

Note that the $\Q_G$--module structures on the last two chain complexes is given by the obvious identification of the fundamental groups
of $M_h$ and $X$, and on $C_{G}(Y)$ the identification is given by $h_*$.  Taking this into account,  \cite[Theorem. 3.1]{mwt} gives:

\begin{lm}\label{lm:deltatau} If $h:Y\to X$ is a homotopy equivalence, then 
$$h_*(\Delta(Y))=\Delta(X)-\tau(h).$$
\end{lm}

A corollary is the following version of \cite[Lemma 12.5]{mwt}:
\begin{cor}\label{cor:delta} Assume $h:Y\to X$ be a homotopy equivalence between spaces as above.\par
If $h$ is a simple homotopy equivalence, then $h_*(\Delta(Y))=\Delta(X)$.\par
If  $Wh(\pi_1(X))$ is torsion free, the converse is also true.
\end{cor}

Now let $W$ be an $h$--cobordism between the $n$-manifolds $M$ and $M'$ with Whitehead torsion $\tau(W,M)=\tau$.
Let $j:M'\subset W$ be the inclusion and let $r:W\to M$ be a (deformation) retraction, and  set $h=r\circ j$. Then the 
Whitehead torsion of the composition $h=r\circ j$ is given by
\begin{equation}\label{eq:tauh}
\tau(h)=-\tau + (-1)^n\bar\tau.
\end{equation}

Using Lemma \ref{lm:deltatau} we get Milnor's formula for the relation between $\tau$ and the $R$--torsions of 
$M$ and $M'$ \cite[p. 405]{mwt}:
\begin{equation}\label{eq:deltatau}
h_*(\Delta(M'))=\Delta(M)+\tau-(-1)^n\bar \tau.
\end{equation}

(Milnor writes this multiplicatively, but we choose an additive notation, emphasizing that this takes place in an abelian group. He
also only gives the formula for $n$ odd.)\smallskip 

From now on we assume $n$ is odd and write $\pi_1(M)=G$.  Recall the result of Wall (\cite[7]{Wu}, see also \cite[Cor. 7.5]{oliver})  that  for
a finite group the standard involution acts trivially on  $Wh'(G)$.  Hence formula (\ref{eq:deltatau}) now reduces to
\begin{equation}\label{eq:delta}
h_*(\Delta(M'))=\Delta(M)+2\tau.
\end{equation}

Since $Wh'(G)$ is a nontrivial, free abelian group if $Wh(G)$ is infinite, it follows that given
$M$ (with the property that the fundamental group acts trivially on the rational homology of its universal cover)   
 we can realize infinitely many elements of the form $h_*(\Delta(M'))$.  We claim that these must represent infinitely many
homeomorphism types of manifolds $M'$.\par

Let $(M_1,h_1)$ and $(M_2,h_2)$ be two such choices, and assume $f:M_1'\to M_2'$ is a homeomorphism.
Then $f_*(\Delta(M_1)=\Delta(M_2)$ and hence ${h_2}_*(\Delta(M_2))=(h_2f)_*(\Delta(M_1))$ --- \ie 
${h_1}_*(\Delta(M_1))$ and ${h_2}_*(\Delta(M_2))$ are both images of the same element under homomorphisms
induced by isomorphisms $\pi_1(M_1')\approx \pi_1(M)$.  But there are only finitely many isomorphisms between given
finite groups.  Thus only finitely many elements ${h_2}_*(\Delta(M_2))$ can be realized by manifolds homeomorphic to
$M_1$. It follows that infinitely many homeomorphism types (in fact, simple homotopy types) occur at the other end of 
 $h$--cobordisms from $M$.\smallskip

It remains to construct such manifolds $M$, given a finite group $G$.  Equivalently,  we want to construct 
simply-connected manifolds $\widetilde M$ with free, $H_*(-;\Q)$-trivial $G$-actions.  Methods
for such constructions have been developed by Pardon \cite{P} and Weinberger \cite{WeI, WeII}.  For the sake of completeness, here
is a short sketch, following \cite {P}:\par
Let $p$ be a prime not dividing the order of $G$. A classical construction of Swan \cite{Sw} gives for every odd $m\geqslant 3$
a finite, simply-connected $CW$-complex $X$ such that (a) $X$ has a  free $G$-action,  and (b) $H_*(X;\Z_{(p)})
\approx H_*(S^m; \Z_{(p)})$.   In fact, from Serre $\mathcal C$-theory there is a $\Z_{(p)}$ homology isomorphisms $f:S^m\to X$ 
realizing (b). and $q:X\to X/G$ (the quotient map).   The quotient $X/G$ is then a $\Z_{(p)}$-Poincar\'e complex of formal dimension $m$, and the quotient map $q:X\to X/G$ is also a $\Z_{(p)}$ homology isomorphism.
Forming the composition $h=q\circ f$ defines a $p$-local normal map $h:(S^m, \nu)\to (X/G,\xi)$ (not necessarily of degree one), 
where $\xi$ is the trivial bundle.\par
Now $h$ has a surgery obstruction in $L^h_m(\Z_{(p)}G)$, hence also in $L^h_m(\Q G)$.  One shows that $L^h_3(\Q G)=0$,  thus rational surgery on $h$ can be completed if $m\equiv 3 \mod 4$.  This gives an $m/2$-connected manifold $\widetilde M^m$ such that
$H_*(\widetilde M^m;\Q)\approx H_*(S^m;\Q)$ with a free $G$-action, and we let $M^m=\widetilde M/G$.\par
For $m\equiv 1\mod 4$, take $N^m=M^{m-2}\times S^2$.  This gives examples in all odd dimensions $\geqslant 7$.  Note 
that  in dimensions $4k+3$ the manifolds are rational homology spheres, and in dimensions $4k+1$ they have rational homology
as $S^{4k-1}\times S^2$.   Weinberger's results \cite[Theorem 4.7]{WeII} generalize this and extend it to dimension 5,  providing 
examples which are rationally $S^3\times S^2$.

\section{Remarks on "special" manifolds}\label{sec:special}

Recall that Milnor \cite[p.\,404]{mwt} called a compact  manifold {\em special} if the fundamental group is finite abelian
and acts trivially on the rational homology groups of its universal covering space. The trivial action is needed for the definition 
of $R$-torsion, but the abelian assumption was based on the incorrect claim  that the Whitehead group of a finite abelian 
group $G$ is free. (See Example 1.6 and the following remark in \cite{mwt}.) 
Or, equivalently, that $SK_1(\Z G)=\ker(K_1(\Z G)\to K_1(\Q G))=0$. This is, of course, 
the case for $G$ cyclic, which was the most important case for Milnor, but the general statements about special manifolds
need the extra hypothesis.  \par
We would now like to redefine
a special manifold to be one with finite fundamental group $G$ whose Whitehead group has no torsion, \ie  $SK_1(\Z G)=0$,
and such that $G$ acts trivially on the rational homology of the universal covering space. Then the 
general  statements about special manifolds in \cite[\S 12]{mwt} hold with suitable modifications, with the same proofs.
The crucial observations are that now $Wh(G)=Wh'(G)\to Wh(\Q_G)$ is an injection, and the converse holds in 
Lemma \ref{cor:delta}. \par
In particular, we have (cf. \cite[Theorem 12.8]{mwt}):

\begin{thm}\label{thm:mwt12.8}
An inertial $h$-cobordism between odd-dimensional special manifolds of dimension $\geqslant 5$ which is compatible with the natural identifications of fundamental groups is a product.
\end{thm}

That the inertial $h$-cobordism $(W;M,M')$ is 'compatible with the natural identifications of fundamental groups' means that 
there is a homeomorphism $f:M'\to M$ such that the map induced on Whitehead groups by $f$ and the natural homotopy equivalence
$h:M'\subset W\to M$ coincide.  Then $h_*(\Delta(M'))=\Delta(M)$ by Corollary \ref{cor:delta}, and  Theorem \ref{thm:mwt12.8} 
follows from (\ref{eq:delta}).\par

\begin{rem}
There are  examples of $h$-cobordisms between special 3-manifolds which are compatible with identifications of fundamental groups,
but not products. In fact, these $h$-cobordisms are even strongly inertial.  The first examples were found by Cappell and Shaneson
in \cite{CS1}.
\end{rem}
  
The assumption of compatibility of identifications of fundamental groups is stronger than being inertial, even in the case of cyclic
fundamental group, as illustrated by the following example (cf. also \cite{m61} and \cite{H80}):
\begin{ex}\label{ex:lens}
Many classical  three-dimensional lens spaces $L^3(p,q)$ admit self-homotopy equivalences $h$ with non-trivial Whitehead torsion
$\tau(h)$.   Choose one with fundamental group of odd order --- the simplest example is $L^3(5,1)$ and $h_*$ equal to multiplication by 2 on $\pi_1\approx\Z/5$ --- and consider the homotopy equivalence 
$$L^3\xrightarrow{h} L^3\xrightarrow{j} L^3\times D^3.$$
Here $D^3$ is the 3-disk and $j$ is the inclusion of $L^3\times\{0\}$.  The map $j\circ h$ is homotopic to a smooth
embedding $i:L^3\to L^3\times D^3$ by \cite{Hae}, and $i(L^3)\subset L^3\times D^3$ has trivial normal disk bundle $N$. (In fact,
 since $[L^3, BO(3)]=0$, every linear $D^3$-bundle over $L^3$ is trivial.)   Then $L^3\times D^3 -\text{int} N$
is an $h$-cobordism with torsion $\tau(h)$, and both boundary components are homeomorphic to $L^3\times S^2$.  
See also \cite[\S6]{H80}. \end{ex}

Note that the manifold $L^3\times S^2$ is still special, but the natural identification of fundamental groups is via $h$, which is not
homotopic to a homeomorphism.\par
 On the other hand, the condition of compatibility of fundamental groups is  weaker than being strongly inertial. 
In fact,  in the next section we will prove the following Theorem:

\begin{thm}\label{thm:abelian}
Suppose $M$ is a closed oriented manifold of odd dimension with abelian fundamental group of finite order.
Then every strongly inertial $h$-cobordism from $M$  is trivial.
\end{thm}  

\begin{ex}\label{ex:K2}
Let $K$ be a finite, 2-dimensional complex such that $\pi_1(K)=\Z/4\times\Z/4$.   For example, take $K$ to be the
2-torus $T^2=S^1\times S^1$ with two 2-disks attached along $S^1\times\{1\}$ and $\{1\}\times S^1$ by attaching 
maps of degree four.  Oliver's calculations \cite{oliver} show that $Wh(\Z/4\times\Z/4)\approx \Z/2$, and by \cite{Lat}
there exists a self-homotopy equivalence $f:K\to K$ with torsion $\tau(f)$ equal to the nontrivial element.\par
Now embed $K$ in Euclidean space $\R^{2k}$, $k\geqslant 3$ and let $N(K)$ be the topological regular neighborhood \cite{J}.
Approximate the composition $K\xrightarrow{h} K\subset N(K)$ by an embedding. By uniqueness of regular neighborhoods 
we now have an embedding $j:N(K)\xrightarrow{\subset} N(K)$, and $W=N(K)-j(\text{int}N(K))$ is then an inertial 
$h$-cobordism from $\partial N(K)$ with non-trivial torsion. Moreover, since there is only one possible isomorphism
between the  Whitehead groups involved, this $h$-cobordism is obviously compatible with the natural identifications of fundamental
groups.  But by Theorem \ref{thm:abelian} it can not be strongly inertial.
\end{ex}

\begin{rem}
Note that the natural (self-) homotopy equivalence $h:\partial N\subset W\to\partial N$
between the ends of the $h$-cobordism constructed in Example \ref{ex:K2} is simple and $h$-cobordant to the identity. 
However, by Theorem \ref{thm:abelian}, it is not homotopic to a homeomorphism.  From the point of view of surgery theory, this 
is an example of a negative answer to the following general and challenging question: "Suppose $f:N\to N$ is a self-homotopy equivalence of a 
compact manifold and $f$ is normally cobordant to the identity. Is $f$ homotopic to a homeomorphism?"  See \cite{CS} and 
\cite{KS95} for specific cases of this problem.
\end{rem}

 \section{Proof of Theorem \ref{thm:abelian}}\label{sec:abelian}

If the dimension of $M$ is 1, the claim is obvious.  If dim\,$M=3$, $M$ must be a lens space, and every $h$-cobordism is 
topologically a product by \cite{KS89}.  Consequently we may assume dim\,$M\geqslant 5$.  We will now use surgery theory and the work of Hambleton, Milgram, Taylor and Williams \cite{HMTW} to analyze
$SI(M)$. \par
A strongly inertial $h$-cobordism $(W;M,N)$ determines an element of the structure set $\SH(M\times I)$.  (By convention, if $V$
has boundary, an element $V'\to V$ of the structure set $\SS(V)$ or $\SH(V)$ is a homeomorphism on the boundary.)
This structure set sits in a diagram of Sullivan--Wall exact surgery sequences ($G=\pi_1(M)$):
$$\xymatrix{
\cdots\ar[r]&  L^s_{n+2}(G)\ar[r]^{\gamma^s}\ar[d]^{l_1} & \SS(M\times I)\ar[r]^{\eta^s}\ar[d]^{t} &
\N(M\times I)\ar[r]^{\theta^s}\ar[d]^= & L^s_{n+1}(G)\  \ar[d]^{l_0}  \\
\cdots\ar[r]&  L^h_{n+2}(G)\ar[r]^{\gamma^h} & \SH(M\times I)\ar[r]^{\eta^h} &
\N(M\times I)\ar[r]^{\theta^h} & L^h_{n+1}(G)\,.
}$$
This is a diagram of abelian groups and homomorphisms, where the addition in the groups in the middle is given by 
"stacking in the $I$-direction".  Moreover, we need the following two highly non--trivial facts:
\begin{enumerate}
\item Since $n+2$ is odd, the map labeled $l_1$ is surjective \cite{bak}.
\item When $n+1$ is even,  $l_0$ is injective on the image of $\theta^s$. (Proof below.)
\end{enumerate}
Given this, half of the standard proof of the five--lemma gives that the map $t:\SS(M\times I)\to \SH(M\times I)$ is surjective. Thus
$W$ is $h$-cobordant rel boundary to another $h$-cobordism with trivial torsion, hence to $M\times I$.  We want to conclude that
$W$ itself must have trivial torsion.\par
Let $(U; W,M\times I)$  be an $h$-cobordism and compute the torsion of the homotopy equivalence $M\subset U$ as a composition
in the two obvious ways:  $M\subset W\subset U$ and $M\subset M\times I\subset U$. Then we get (in $Wh(U)$)
$$ \tau(W\subset U)+j_*(\tau(M\subset W))=\tau(M\times I\subset U)+j'_*(M\subset M\times I)=\tau(M\times I\subset U),$$
where $j$ and $j'$ are the respective inclusion maps. But the duality formula gives
$$  \tau(W\subset U)=(-1)^{n+1}\overline{\tau(M\times I\subset U)}=\tau(M\times I\subset U),$$
since $n$ is odd and the involution on $Wh(G)$ is trivial when $G$ is abelian \cite{bak2}. It follows that $j_*(\tau(M\subset W))=0$.
But $j_*$ is an isomorphism, and $\tau(M\subset W)=0$ if and only if  $\tau(W,M)=0$.\par
It remains to establish assertion (2) above.  In other words; we need to prove that if $n+1$ is even and $\theta^h(x)=0$,
then also $\theta^s(x)=0$.  We are grateful to Ian Hambleton \cite{Ham}, who showed us how  calculations sketched by 
Taylor and Williams in the unpublished preprint \cite{TW} could be used to prove what is needed.\par

Recall that Theorem A of \cite{HMTW}  gives explicit formulas for the surgery obstructions $\theta^h$ of the form
$$\theta^h(x)=\theta_0(x)+\kappa^h_m(c(x)),\ m=1,2,3,4,$$
where $m\equiv n+1 \mod 4$.  Here  $\kappa^h_m:H_m(G;\Z/2)\to L^h_{n+1}(G)$ are universally defined homomorphisms and $\theta_0$ is the simply--connected surgery obstruction (index or Kervaire invariant if $n+1$ is even and trivial otherwise).  The element $c(x)$ is the image of $x$ under a certain homomorphism
$\N(M\times I)\to H_2(BG;\Z/2)$.\par
Taylor and Williams give similar formulas 
for $\theta^s$, provided that the finite group $G$ has abelian 2-Sylow subgroup \cite[Theorem 1.2]{TW}.
 More precisely, the homomorphisms $\kappa^h_m$ factor as
$$ H_m(G;\Z/2)\xrightarrow{\kappa^s_m} L^s_{n+1}(G) \xrightarrow{l_0} L^h_{n+1}(G),$$
and $\theta^s$ is given by
$$\theta^s(x)=\theta_0(x)+\kappa^s_m(c(x)).$$

Moreover, they show that in this case $\kappa^s_4$ is trivial. Hence (2) holds if $n+1\equiv 2\mod 4$,  since then $\theta^s$ and $\theta^h$ are equal.\par
  If $n+1\equiv 0\mod 4$ we need the more precise calculations of $\kappa_2^s$ and $\kappa_2^h$ in Section 5 of 
\cite{TW}.
We want to show that if $\kappa^h_2(y)=0$ then $\kappa^s_2(y)=0$.  By a well--known observation of 
Wall \cite[Thm. 12]{Wf},
it suffices to do this for the 2--Sylow subgroup $A$ of $G$.  Let $E\subset A$ be the elementary 2--group consisting of all
the elements of order two. By \cite[Theorem. 5.5]{TW} the sequence
$$H_2(BE;\Z/2)\to H_2(BA;\Z/2)\xrightarrow{\kappa^h_2} L^h_0(A)$$
is exact.  Hence,  if $\kappa^h_2(y)=0$, then $y$ comes from $H_2(BE;\Z/2)$. \par
Now consider the diagram 
$$\xymatrix{H_2(BE;\Z/2)\ar[r]^<<<<<{\kappa_2^s}\ar[d]  & L_0^s(E)\ar[r]^{l_0}\ar[d]& L_0^h(E)\ \ar[d] \\
H_2(BA;\Z/2)\ar[r]^<<<<<{\kappa^s_2}  & L_0^s(A)\ar[r]^{l_0}& L_0^h(A),
} $$
defined by naturality. The assertion now follows from \cite[Lemma 5.2]{TW}, which says that
$H_2(BE;\Z/2)\xrightarrow{\kappa_2^s}L_0^s(E)$ is trivial.\qed\smallskip

Observe that the conditions in Theorem \ref{thm:abelian} are only on the dimension and fundamental group of $M$, hence they
remain satisfied after connected sums with or products with simply connected manifolds.  It follows that in general $SI(M)$ and 
$I(M)$ behave very differently:

\begin{cor}
If dim $M$ is odd, $\pi_1(M)$ is abelian and $Wh(\pi_1(M))$ is nontrivial, then $SI(M\#_2(S^p\times_2( S^{n-p}))\ne 
 I(M\#_2(S^p\times S^{n-p}))$.
\end{cor}
In fact, in most cases one copy of $S^p\times S^{n-p}$ is enough.  For another example, take products $L^3\times S^2$ as
in example \ref{ex:lens}
\par
One could ask if $SI(M)$ is always trivial if $\dim M$ is odd and $\geqslant 5$, but this is not the case.  We observed 
in section \ref{sec:general} that if $\dim M$ is odd, then $SI(M)$ contains all elements of the 
form $\tau-\bar\tau$ (formula \ref{eq:inertinc}), and 
Oliver has constructed groups with non-trivial involutions on the Whitehead groups.  (See \cite[Proposition 24]{oliver80} and
\cite[Example 8.11]{oliver}.) These groups are then also examples of groups $G$ with the property that  every orientable manifold  
with fundamental group $G$ has non-trivial inertial set, as promised in the introduction. \par
We end this discussion with an interesting question suggested by Theorem \ref{thm:abelian}: is $SI(M)$ always homotopy invariant? Or perhaps depending on  $\pi_1(M)$ only?  The
full inertial set $I(M)$ is not, as shown by Hausmann \cite[Theorem 6.6]{H80}.

\section{Remarks on categories}\label{sec:diff}

Although everything is formulated in the topological category TOP, the theory discussed in this paper works equally well in  the categories DIFF and PL.
In fact, taking this into account, some of the results are actually slightly stronger than stated.  For instance, in Theorem \ref{thm:noninertial} the manifolds $M$ and $N$ can be chosen to be smooth, but
$M\#(S^p\times S^{n-p})$ and $N\#(S^p\times S^{n-p})$ are not only non--diffeomorphic; they are not even 
homeomorphic.  We should also point out that although the proof of Theorem \ref{thm:abelian} used results from 
\cite{HMTW} and \cite{TW} which deal with  the topological surgery obstruction,  in dimensions $\geqslant 5$ the conclusion of the theorem will automatically be true also in PL and DIFF.  This is because the part of the surgery sequences used in the proof are 
sequences of abelian groups in all categories, and if the homomorphism $l_0$ is injective on the image of the topological
surgery obstruction, it is also injective on the image of the smooth and PL surgery obstructions. 
\par
To compare the statements concerning inertial sets we need the following:

\begin{lm}
Let $M$ be a smooth manifold of dimension $\geqslant 5$.  If $W$ is a topological $h$-cobordism from $M$, then $W$ has
a unique smooth structure compatible with the given structure on $M$.  If in addition $W$ is inertial in TOP, then it is also
inertial in DIFF.\par
Similarly when $M$ is a PL manifold.
\end{lm}
\begin{proof}
The first part follows immediately from smoothing theory \cite{HM}, \cite{KiSi}. 
To prove the second assertion, we use a well-known trick to prove that if $M$ and $N$ are smoothly $h$-cobordant, 
then $M\times \R$ and $N\times\R$ are diffeomorphic (cf. \cite{S}).\par 
 Let $N$ be the other boundary component
of $W$, and let $-W$ be the $h$-cobordism from $N$ to $M$ such that $W \cup_N (-W)\approx M\times I$. Then it also follows that 
$-W \cup_M W\approx N\times I$.  Stacking infinitely in both directions then gives CAT homeomorphisms
\begin{multline*}
M\times\R\approx\cdots (W \cup_N -W)\cup_M(W \cup_N -W)\cdots\\
=   \cdots (-W \cup_M W)\cup_N(-W \cup_M W)\cdots\cdots \approx N\times\R.
\end{multline*}
But then $M$ and $N$ are also diffeomorphic, by the product structure theorem in smoothing theory \cite{HM}, \cite{KiSi}.
\end{proof}


\end{document}